\documentclass[11pt]{amsart}

\usepackage[a4paper,hmargin=3.5cm,vmargin=3.5cm]{geometry}
\usepackage{amsfonts,amssymb,amscd,amstext}
\usepackage{graphicx}
\usepackage[dvips]{epsfig}

\usepackage{fancyhdr}
\usepackage{times}

\usepackage{enumerate}
\usepackage{titlesec}
\usepackage{mathrsfs}

\def\M{\mathbb{M}}

\def\h{\mathbb{H}}
\def\r{\mathbb{R}}
\def\n{\mathbb{N}}
\def\s{\mathbb{S}}

\titleformat{\section}
{\filcenter\bfseries\large} {\thesection{.}}{0.2cm}{}
\titleformat{\subsection}[runin]
{\bfseries} {\thesubsection{.}}{0.15cm}{}[.]
\titleformat{\subsubsection}[runin]
{\em}{\thesubsubsection{.}}{0.15cm}{}[.]

\usepackage[up,bf]{caption}

\newtheorem{theorem}{Theorem}[section]

\newtheorem{corollary}[theorem]{Corollary}
\newtheorem{remark}[theorem]{Remark}

\theoremstyle{definition}

\numberwithin{equation}{section}
\numberwithin{figure}{section}

\usepackage{color}

\begin{document}

\fancyhead[CO]{Stable constant mean curvature surfaces} 
\fancyhead[CE]{ R. Souam} 
\fancyhead[RO,LE]{\thepage} 

\thispagestyle{empty}

\vspace*{1cm}
\begin{center}
{\bf\LARGE Stable  constant mean curvature surfaces with free boundary  in slabs }

\vspace*{0.5cm}

{\large\bf Rabah Souam}
\end{center}

\noindent Institut de Math\'{e}matiques de Jussieu-Paris Rive Gauche,  CNRS,  UMR 7586, B\^{a}timent Sophie Germain,  Case 7012, 75205  Paris Cedex 13, France.

\noindent e-mail: {\tt rabah.souam@imj-prg.fr}

\vspace*{0.1cm}

\vspace*{1cm}

\begin{quote}
{\small
\noindent {\bf Abstract}\hspace*{0.1cm}\vspace*{0.2cm} 
We study stable constant mean curvature (CMC) hypersurfaces $\Sigma$ with free boundary in slabs in a  product space $M\times\r,$ where $M$ is an orientable Riemannian manifold. We obtain a characterization of stable cylinders and prove that if $\Sigma$  is not a cylinder then it is locally a vertical graph. Moreover, in case $M$ is $\h^n,\r^n$ or $\s_+^n,$ if  each  component of $\partial\Sigma$ is embedded,  then  $\Sigma$ is rotationally invariant. When $M$ has dimension 2 and Gaussian curvature bounded from below by a positive constant $\kappa,$ we prove there is no stable CMC with free boundary connecting the boundary components  of a slab of width $l\geq 4\pi/\sqrt{3\kappa}.$ 
We also show that a stable capillary surface of genus 0 in  a warped product $[0,l]\times_f M$ where 
$M=\r^2, \h^2$ or $\s^2,$ is rotationally invariant. Finally, we prove that a stable CMC  immersion of a closed  surface 
in $M\times\s^1(r),$ where $M$ is a  surface with Gaussian curvature bounded from below by a positive constant $\kappa$ and $\s^1(r)$ the circle of radius $r,$  lifts to $M\times\r$ provided 
$r\geq4/\sqrt{3\kappa}.$

\vspace*{0.2cm}

\noindent{\bf Keywords}\hspace*{0.1cm}  constant mean curvature, free boundary, capillary surfaces, stability.

\vspace*{0.2cm}

\noindent{\bf Mathematics Subject Classification (2010)}\hspace*{0.1cm} 53A10, 49Q10, 53C42, 76B45.}

\end{quote}


\section{Introduction }\label{sec:intro}

Let $W$ be an oriented smooth Riemannian manifold of dimension $n\geq 3$ and $\mathcal B\subset W$ a closed domain with smooth boundary. We consider  constant mean curvature (CMC) hypersurfaces with free boundary in $\mathcal B,$ that is, CMC hypersurfaces  whose interior is contained  in $\mathcal B$ and whose boundary lies on $\partial\mathcal B $ and which meet  orthogonally   $\partial\mathcal B.$   They are stationary for the area functional for variations preserving the {\it enclosed volume}. It is  interesting to study  CMC hypersurfaces with free boundary which are stable, that is, those for which the second variation of the area is nonnegative for all volume-preserving variations. This problem arises naturally when 
studying the isoperimetric problem in  $\mathcal B.$  More generally,  a capillary hypersurface in  $\mathcal B$  is a CMC hypersurface in $\mathcal B$ with boundary on $\partial \mathcal B$ and meeting $\partial\mathcal B$ at a constant angle. Capillary hypersurfaces are also stationary for a functional which is a linear combination of  the area of the hypersurface and the area of the {\it domain enclosed} by its boundary on $\partial \mathcal B,$ for volume-preserving variations. The study of stability of capillary hypersurfaces in balls  or with planar boundaries in space forms has attracted a lot of attention very recently, see for instance  \cite{A-S, barbosa, choe-koiso, li-xiong, nunes, wang-xia}.

We are mainly interested in this paper in  stable CMC hypersurfaces with free boundary in slabs of a  product space $M\times\r,$  where $M$ is an orientable  Riemannian manifold  of dimension $\geq 2.$ A slab in $M\times\r$ is the region between two slices $M\times\{t\}.$ Without loss of generality, a slab of width $l>0$ in $M\times\r$ will be taken to be   the domain  $M\times[0,l].$ The special case $M=\r^n$ was considered by Athanassenas \cite{athanassenas} and Vogel \cite{vogel}, for $n=2$,  and by Pedrosa and Ritor\'e \cite{pedrosa-ritore},  and by Ainouz and Souam \cite{A-S} for any $n\geq 2.$ 

The simplest examples of CMC hypersurfaces with free boundary in a slab $M\times[0,l]$ are obtained by taking cylinders above closed CMC hypersurfaces in $M.$  Our first result   
(Theorem \ref{product}) gives   a characterization of stable cylinders in terms of the first eigenvalue of the stability operator of the base and the width of the slab. When $M$ is the Euclidean space $\r^n,$ the hyperbolic space $\h^n$ or the unit sphere $\s^n,$ this characterization of stable tubes is  explicit  in terms of their radius and the width
(Corollary \ref{tubes}). 

In the general case, we show (Theorem \ref{thm: free boundary}) that a stable free boundary CMC hypersurface in $M\times [0,l]$ that is not a cylinder has to be locally a vertical graph. It is moreover a global graph in case each of its boundary components is embedded and $M$ is simply connected. When $M$ is the Euclidean space $\r^n,$ the hyperbolic space $\h^n$ or 
a hemisphere $\s_+^n,$   the previous hypotheses imply that the hypersurface is rotationally invariant. This extends results obtained in \cite{A-S} in the Euclidean case.

When $M$ is two-dimensional with curvature bounded from below by a positive constant $\kappa,$ we prove (Theorem \ref{thm:nonexistence}) that no stable CMC surface with free boundary can connect 
the boundary components of a slab $M\times[0,l]$ whose width satisfies $l \geq  4\pi /\sqrt{3\kappa}.$ 

In section \ref{sec:warped}, we consider  stable capillary surfaces in warped products of the type $[0,l]\times_f M$ where $M=\r^2, \h^2$ or $\s^2.$ We prove they have to be rotationally invariant when they have genus zero, extending a result obtained in \cite{A-S} in the Euclidean case. 
This applies, in particular, to the region bounded by two parallel horospheres in the hyperbolic space $\h^3$ (Corollary \ref{cor:horospheres} ).

Finally, in Section \ref{sec:closed}, we  obtain a result of the same type as Theorem \ref{thm:nonexistence} for closed stable CMC surfaces in the product $M\times\s^1(r)$ of an orientable surface $M,$ with curvature bounded from below by a positive constant $\kappa,$ by the circle $\s^1(r)$ of radius $r>0.$ More precisely, we show that if $r\geq4/\sqrt{3\kappa},$  then the immersion lifts to a stable CMC immersion in $M\times\r$ (Theorem \ref{thm:closed}). 


\section{Preliminaries }\label{sec:preliminaries}

Let $W$ be an oriented smooth Riemannian manifold of dimension $n\geq 3$ and $\mathcal B\subset W$ a closed domain with smooth boundary. A capillary hypersurface in $\mathcal B$ is a compact and constant mean curvature (CMC)  hypersurface, with non-empty boundary,  which  is contained in $\mathcal B$  and  meets $\partial\mathcal B$ at a constant angle along its boundary.  When the angle of contact  is $\pi/2$, that is, when the hypersurface  is orthogonal to $\partial \mathcal B$, the hypersurface is said to be a CMC hypersurface  with free boundary in $\mathcal B$. 
Capillary hypersurfaces are critical points of an energy functional for volume-preserving variations. In the free boundary case the energy functional is the area functional. We refer 
to \cite{A-S} for more details.

Consider a  capillary hypersurface defined by a smooth immersion $\psi: \Sigma\longrightarrow \mathcal B$ and  let $N$ be a global  unit normal to $\Sigma$ along $\psi$ chosen so that its (constant) mean curvature satisfies $H\geq 0.$ Denote by $\overline N$ the exterior unit normal to $\partial\mathcal B.$ The angle of contact is 
the angle $\theta \in (0,\pi)$  between $N$ and $\overline N.$ We denote by $\nu$ the exterior unit normal to $\partial\Sigma$ in $\Sigma$ and by $\bar\nu$ the unit normal to $\partial\Sigma$ in $\partial\mathcal B$ so that 
$\{N,\nu\}$ and $\{\overline N, \bar\nu\}$ have the same orientation in $(T\partial\Sigma)^{\perp}.$ The angle between  $\nu$ and $\bar\nu$ is also equal to $\theta.$

 The index form $\mathcal I$ of $\psi$ is the symmetric bilinear form defined on 
 the first Sobolev space $H^1(\Sigma)$  of $\Sigma$   by
\[ \mathcal I (f,g)=\int_{\Sigma}\left( \langle \nabla f, \nabla g\rangle -(|\sigma|^2+\text{Ric}(N)) fg\right) d\Sigma-
 \int_{\partial\Sigma}q\, fg\,d(\partial\Sigma),
\]
where $\nabla$  stands for the gradient for the metric induced by $\psi,$  $\sigma$ is the second fundamental form of $\psi,$ Ric$(N)$ is the Ricci curvature of $W$ in the direction $N$ and  
\[ q= \frac{1}{\sin\theta} \text{II} (\bar\nu,\bar\nu)+ \cot\theta \,\sigma(\nu,\nu). 
\]
Here II denotes the second fundamental form of $\partial\mathcal B$ associated to the unit normal $-\bar N.$  

  Let $\mathcal F=\{ f\in H^1(\Sigma), \int_{\Sigma} f\,d\Sigma =0\}.$ The immersion  $\psi$ is capillarily stable if  $\mathcal I(f,f)\geq 0$ for all $f\in\mathcal F.$ 

The stability operator on $\Sigma$ is the linear operator defined on $\mathcal C^2(\Sigma)$ by 
$$ L= \Delta +|\sigma|^2+\text{Ric}(N)$$
where $\Delta$ is the Laplacian on $\Sigma$ induced by $\psi.$

A function $f\in \mathcal F$ is said to be a Jacobi function of $\psi$ if it lies in the kernel of $\mathcal I$, that is, if 
$\mathcal I(f,g)=0$ for all $g\in \mathcal F.$ By standard arguments, this is equivalent to saying that $f\in \mathcal C^{\infty}(\Sigma)$ and satisfies
\begin{align*}
L f &= \text{constant \quad on }\, \Sigma\\
\frac{\partial f}{\partial \nu} &=q\,f\quad\,\, \quad\quad\text{on}\,\, \partial \Sigma
\end{align*}

We will also need to consider compact immersed CMC hypersurfaces without boundary. Given such an immersion $\psi:\Sigma\longrightarrow W$ in an orientable Riemannian manifold $W,$ it is said to be stable if $\mathcal I (f,f)\geq 0$ for all $f\in \mathcal F= \{f\in H^1(\Sigma),\, \int_{\Sigma} f\,d\Sigma =0\}$ where $\mathcal I$ is the index form defined on $H^1(\Sigma)$ by 
\[ \mathcal I (f,g)=\int_{\Sigma}\left( \langle \nabla f, \nabla g\rangle -(|\sigma|^2+\text{Ric}(N)) fg\right) d\Sigma.
\]
 An immersion $\psi$ of a closed or capillary CMC is said to be  {\it strongly stable} if $\mathcal I(f,f)\geq 0$ for all $f\in H^1(\Sigma).$ 
 
 It is also interesting to consider, as a more general setting for capillarity, the  case where   the contact  angle is constant along each component of $\partial \Sigma,$ allowing it to vary from one component to the other. All the previous discussion 
 on stability extends to this more general case (cf. \cite{A-S}).


{\section{Stable CMC hypersurfaces  with free boundary in a slab in a product space $M\times\r$}}

Let $M$ be a connected  oriented Riemannian manifold of dimension $n\geq 2.$ We discuss in this section  the case of CMC immersed hypersurfaces with free boundary in slabs $M\times[0,l], l>0,$ of the  product space $M\times\r.$
The function $q$ vanishes identically and the index form takes a simple form:
\[ \mathcal I (f,g)=\int_{\Sigma}\left( \langle \nabla f, \nabla g\rangle -(|\sigma|^2+\text{Ric}(N)) fg\right) 
d\Sigma.\]

The simplest examples of  free boundary CMC hypersurfaces in a slab $M\times [0,l]$ are cylinders over closed CMC hypersurfaces in $M.$ We will discuss  the stability of 
cylinders. For this, we will need a criterion for stability of closed CMC hypersurfaces. It  was established by   Koiso (\cite{koiso}, Theorem 1.3) for compact CMC surfaces  with boundary in $\mathbb R^3$ but the
 result readily extends to  the closed case in general  ambient manifolds.

\begin{theorem}\label{criterion}
Let $\psi:\Sigma\longrightarrow M$ be a  closed  and two-sided immersed  CMC hypersurface  in a  Riemannian manifold  $M$  of dimension $n\geq 2$  .  
Denote by $\lambda_1(\Sigma) $ and $\lambda_2(\Sigma),$ 
respectively, the first and second eigenvalues of the stability  operator $L$ on $\Sigma$. 
The following hold 
\smallskip

(\text{I}) $\lambda_1(\Sigma)\geq 0$ if and only if $\psi$ is strongly stable.

\smallskip

(\text{II})  If $\lambda_1 (\Sigma) <0$ and $\lambda_2(\Sigma)>0,$ then there exists a uniquely determined smooth function $u$ on $\Sigma$ such that $Lu=1$ and $\psi$ is stable if and only if $\int_{\Sigma} u\,d\Sigma \geq 0.$

\smallskip

(\text{III})  If  $\lambda_1 (\Sigma)<0$ and $\lambda_2(\Sigma)=0$
 and  there exists an eigenfunction associated to $\lambda_2(\Sigma)$ satisfying $\int_{\Sigma} g \, dA \neq 0,$ then $\psi$ is unstable.

\smallskip

 (\text{IV})  If  $\lambda_1(\Sigma)<0$ and $\lambda_2(\Sigma)=0$
 and  $\int_{\Sigma} g \, dA =0$ for any eigenfunction $g$ associated to $\lambda_2.$
Denoting by $E$ the eigenspace 
associated to $\lambda_2(\Sigma)=0,$  there exists a uniquely determined smooth function $u\in E^{\perp}$ (the orthogonal to $E$ in the $L^2$ sense) which satisfies $Lu=1.$
Then $\psi$ is stable if and  only if $\int_{\Sigma} u\,  dA\ge 0.$ 

\smallskip
(\text{V})  If $\lambda_2(\Sigma) <0,$ then $\psi$ is unstable. 
\end{theorem}

\begin{remark}\label{remark1}
If $\lambda_1(\Sigma)<0$ and $\lambda_2(\Sigma)\geq 0$ and we know beforehand the existence of a smooth function $u$ satisfying $Lu=1,$ then, it is easily checked that (II)  and (IV) can be restated simply by saying that $\psi$ is stable if and only if $\int_{\Sigma} u\, dA\geq 0.$
 \end{remark} 
 
We can now give the following characterization of  stable cylinders.

\begin{theorem}\label{product}
Let $\psi:\Gamma\to M$ be an  immersion with constant mean curvature $H$  of a closed oriented  manifold $\Gamma$ of dimension  $n\geq 1$ 
in an $(n+1)$-dimensional oriented manifold $M$ and let $l>0.$ The map $\overline\psi$  $$\overline\psi:=\psi\times\{\text{id}\}:\Gamma \times[0,l] \longrightarrow M\times[0,l],$$
is an immersion of  constant mean curvature $\frac{(n-1)}{n}H$  with free boundary. Denote by $\lambda_1(\Gamma)$ the first eigenvalue of the stability operator $L$ on $\Gamma.$ The following hold
  
  (i)  If  $\psi$ is unstable, then $\overline{\psi}$ is unstable too.

(ii) If   $\psi$ is  stable, then $\overline{\psi}$ is  stable if and only if 
\begin{equation}\label{cylinder}
\lambda_1(\Gamma) +\left( \frac{\pi}{l}\right)^2\geq 0.
\end{equation}
\end{theorem}
\begin{proof}
The first statement about $\overline\psi$ is straightforward.

 Denote by $\overline L $  the stability operator of  $\overline \psi$ and by   $\overline{\mathcal I}$    the associated index  form.  
 
$(i)$ Suppose $\psi$ is unstable, then there exists   $u\in H^1(\Gamma)$ such that  $\int_{\Gamma} u=0$ and   ${\mathcal I}(u,u)<0.$ Define $\bar u$  on $\Gamma\times[0,l]$ by $\bar u (p,t)=u(p),$ for all $(p,t)\in \Gamma\times[0,l].$ Then clearly $\bar u\in  H^1(\Gamma\times[0,l]),\quad \int_{\Gamma\times[0,l]} \bar u =0 $ and $\overline {\mathcal I} (\bar u, \bar u)={\mathcal I} (u,u)<0,$
which shows that 
$\overline\psi$  is unstable. 

$(ii)$  Suppose $\psi$ is stable. 

If $\lambda_1(\Gamma) +\left( \frac{\pi}{l}\right)^2<0,$ let $u$ be an eigenfunction of $L$ associated to $\lambda_1(\Gamma).$ Then the function  $v:=u\, \cos{(\pi t/l)} $
on $\Gamma\times [0,l]$ satisfies $\int_{\Gamma\times[0,l]} v=0$ and $\overline{\mathcal I}(v,v)=(\lambda_1(\Gamma) +\left( \frac{\pi}{l}\right)^2) \int_{\Gamma\times[0,l]} v^2<0.$ So $\overline\psi$ is unstable.  

Suppose now that $\lambda_1(\Gamma) +\left( \frac{\pi}{l}\right)^2\geq 0.$
 We will first show that the CMC 
immersion $\widehat\psi:= \psi\times \{\text{id}\}: \Gamma\times\s^1\left(\frac{l}{\pi}\right)\longrightarrow M\times\s^1\left(\frac{l}{\pi}\right)$ is stable, where $\s^1\left(\frac{l}{\pi}\right)$ denotes the circle of radius $ \frac{l}{\pi}.$

Let  Ric and 
 $\widehat{\text{Ric}},$ be the Ricci curvature tensors of $M$ and $M\times \r,$ respectively, and denote by $\sigma,$  and $\widehat\sigma,$ the 
second fundamental forms of $\psi$ and $\widehat\psi,$ respectively.  We note that $|\widehat \sigma|=|\sigma| $ and if $N$ is  a unit field normal to $\Gamma$ along $\psi$, then it is also normal to $\Gamma \times \s^1\left(\frac{l}{\pi}\right)$ along $\widehat \psi$  and $\widehat{\text {Ric}}(N,N)={\text {Ric}}(N,N).$ We denote by $\widehat L $ be the stability operator of  $\widehat \psi$ and let   $\widehat {\mathcal I}$  be  the associated index  form.

 The eigenvalues and eigenfunctions of the eigenvalue  problem  
 $$ f^{\prime\prime} +\mu \, f =0$$
 on $\s^1\left(\frac{l}{\pi}\right)$  are given by  $\mu_n= (n\pi/l)^2,\, n=0,1,2,\dots$ and  $f_n(t)=\cos{\sqrt{\mu_n}} t.$
 
 Let $\lambda_1(\Gamma)<\lambda_2(\Gamma)\leq ....\leq \lambda_k(\Gamma)\leq...$ be the sequence of eigenvalues of the operator $L$ on the closed manifold 
$\Gamma. $ As the function $|\widehat\sigma|^2+\widehat{\text{Ric}}( N, N)=|\sigma|^2+{\text{Ric}}(N,N)$ depends only on the first factor of the product  $\Gamma\times \s^1\left(\frac{l}{\pi}\right)$, a standard argument shows that the eigenvalues of the operator 
$\widehat L$ on $\Gamma\times \s^1\left(\frac{l}{\pi}\right)$ are given by
$\lambda_m(\Gamma)+
n^2\pi^2/l^2,\, m, n\in \n.$ 

In particular, the first and second eigenvalues of $\widehat L$ on $\Gamma\times\s^1\left(\frac{l}{\pi}\right)$ verify: 
\begin{equation}\label{lambda1}
 \lambda_1\left(\Gamma\times\s^1\left(\frac{l}{\pi}\right)\right)= \lambda_1(\Gamma)
 \end{equation}
\begin{equation}\label{lambda2}
\lambda_2\left(\Gamma\times\s^1\left(\frac{l}{\pi}\right)\right)=\min \{ \lambda_1(\Gamma)+ \frac{\pi^2}{l^2}, \lambda_2(\Gamma)\}
\end{equation}

 If $\lambda_1(\Gamma)\geq 0,$ then, by (\ref{lambda1}), $\lambda_1(\Gamma\times\s^1\left(\frac{l}{\pi}\right))\geq 0,$ that is, $\widehat \psi$ is strongly stable. So we assume $\lambda_1(\Gamma)<0.$ In this case we know 
that $\lambda_2(\Gamma) \geq 0$ (Theorem \ref{criterion} (V)). Therefore, by (\ref{lambda2}), 
$\lambda_2\left(\Gamma\times\s^1\left(\frac{l}{\pi}\right)\right)\geq 0.$

Consider then the solution $u\in\mathcal C^{\infty}(\Gamma)$ to 
the equation $Lu=1$ described in Theorem \ref{criterion}, (II) or (IV). 
 Then, the function $\hat u,$ defined on $\Gamma\times\s^1(\frac{l}{\pi})\,$ by $\hat u(p,t)=u(p),$ for $(p,t)\in\Gamma\times\s^1(\frac{l}{\pi}),$  verifies $\widehat{L} \hat u =L u= 1$
and $\int_{\Gamma\times\s^1(\frac{l}{\pi})} \hat u = 2l\int_{\Gamma} u \geq 0.$  Theorem \ref{criterion} (and Remark \ref{remark1})  shows that $\widehat\psi$ is stable.

 We now  show that stability of $\widehat\psi$ implies stability of $\overline\psi.$ 

Let $u\in H^1(\Gamma\times[0,l])$ satisfying $\int_{\Gamma\times[0,l]} u =0.$ We extend $u$
to a function in $H^1(\Gamma\times[0,2l])$ by reflection across $M\times\{l\}$, that is, we set $u(p,t)=u(p,2l-t)$ for all $p\in\Gamma$ and 
$t\in[l,2l].$ Clearly this function gives rise to a function $\hat u \in H^1\left(\Gamma\times\s^1\left(\frac{l}{\pi}\right)\right)$ satisfying 
$\int_{\Gamma\times\s^1(\frac{l}{\pi})}\hat u =0$ and $\mathcal I (u,u)= \frac{1}{2}\widehat{\mathcal I }(\hat u, \hat u)\geq 0.$ Therefore $\overline \psi$ is stable. This completes the proof. 
 \end{proof}

For instance, consider the case of a tube of radius $\rho$ in   ${\M}^n(\kappa)\times [0,l],\,\kappa =-1,0,1,\, n\geq 2,$ where ${\M}^n(-1)=\h^n,{\M}^n (0)=\r^n$ and  ${\M}^n (1)=\s^n.$  The stability operator of the tube writes 
$$L= \Delta +(n-1)(c_{\kappa}(\rho)^{-1} + \kappa)$$
with 
 \begin{align*}
c_{-1}(\rho)&= {\tanh^2 \rho} \\
c_0(\rho)\,\,&= \rho^2\\
c_{1}(\rho)\,\,&={\tan^2 \rho} 
\end{align*}

As round spheres in space forms are stable, we get the following  characterization of stable tubes in $\M^n(\kappa)\times\r,$ the case of tubes in $\r^3$ was treated  by Athanassenas \cite{athanassenas} and Vogel \cite{vogel}.

\begin{corollary}\label {tubes}
A  tube of radius $\rho$ and height $l$ in $\h^n\times [0,l]$ is stable if and only if  $\pi \sinh \rho \geq \sqrt{n-1}\, l.$

A  tube of radius $\rho$ and height $l$ in $\r^n\times [0,l]$ is stable if and only if  $\pi  \rho \geq \sqrt{n-1}\, l.$

A  tube of radius $\rho$ and height $l$ in $\s^n\times [0,l]$ is stable if and only if $\pi \sin \rho \geq \sqrt{n-1} \,l.$
 \end{corollary}

We now prove a general fact about the geometry of stable  free boundary CMC hypersurfaces in a slab in $M\times\r.$ It extends the result proved in \cite{A-S} for the case $M=\r^n$.

\begin{theorem}\label{thm: free boundary} Let $M$ be an $n-$dimensional connected and oriented Riemannian manifold and $\psi:\Sigma\to M\times [0,l]$  a stable free boundary 
immersion connecting  the two boundary components $M\times\{0\}$ and $M\times\{l\}, \, l>0,$ of a slab in the Riemannian product $M\times\r.$ 
 
 Then, either $\psi(\Sigma)$ is a vertical cylinder over a closed immersed  stable CMC in $M$ or   $\psi(\Sigma)$ is  locally a vertical graph and in this case the lowest eigenvalue 
 of the stability operator on $\Sigma$  with Dirichlet boundary condition is equal to 0.
  
  Furthermore, suppose $\psi(\Sigma)$ is not a vertical cylinder. If $M$ is simply connected and   the restriction of $\psi$ to each component of $\partial\Sigma$ is an embedding, then $\psi(\Sigma)$  is a global vertical graph over a domain in $M.$ 
In particular, if $M$ is either the Euclidean space $\r^n,$ the hyperbolic space 
 $\mathbb H^n$ or a
 hemisphere $\mathbb S^n_+,$ then, under the same hypotheses, $\psi(\Sigma)$ is rotationally invariant around an axis $\{p\}\times \r$, for some $p\in M.$ 
 \end{theorem} 

\begin{proof}
  We start with the following general fact. Let $\psi:\Sigma\to M$ be a CMC immersion of an orientable  manifold $\Sigma$ of dimension $n$ in an orientable Riemannian manifold $M$ of dimension $n+1$ and $N$ a global unit normal to $\Sigma.$ If $X$ a Killing vector field on $M$ then the function $\langle X\circ \psi, N\rangle$ lies in the kernel of the Jacobi operator $L$ 
  of $\Sigma.$ We include a proof of this fact for the reader's convenience. The Jacobi operator is also known to be  the linearized operator  of the mean curvature functional. More precisely,  let  $\psi_t$ be a smooth deformation,    through immersions, of $\psi_0=\psi$ and let $H_t$ be the mean curvature function of $\psi_t,$  then:
\begin{equation}\label{killing}  2{\frac{\partial H_t}{\partial t}}|_{t=0}= L(\langle {\frac{\partial{\psi_t}}{\partial t}}|_{t=0},N\rangle).
\end{equation}

Now, let  $\phi_t$ be the local flow of isometries of $M$ generated by the Killing field $X$  and consider the deformation $\psi_t=\phi_t\circ \psi$ of $\psi.$ Clearly, $H_t=H$ for all $t$ and so the left-hand side of (\ref{killing}) vanishes. Furthermore, $ {\frac{\partial{\psi_t}}{\partial t}}|_{t=0}=X\circ\psi.$ This shows that the function $\langle X\circ \psi, N\rangle$ lies in the kernel of $L.$

   In our case, as the vertical unit vector field 
$\frac{\partial}{\partial t}$ is a Killing field of the product $M\times\r,$  
 the function $v=\langle \frac{\partial}{\partial t},N\rangle$ verifies $Lv=0.$ 
 
If  $v$ is identically zero, then $\psi(\Sigma)$ is a vertical cylinder whose base is an   immersed CMC hypersurface in 
$M$ and which is moreover stable by Theorem  \ref{product}. 

 Assume $v$ is not identically zero. We will show it has a sign in the interior of $\Sigma.$
    Suppose by contradiction  that $v$ changes sign and  consider the functions  $ v_+=\text{max} \{v,0\} $ and
 $v_-= \text{min} \{v,0\}$ which lie in the Sobolev space $H^1(\Sigma).$   Using the fact that $v=0$ on $\partial\Sigma,$ one has
 \begin{align*}
 \mathcal I(v_+,v_+)&= \int_{\Sigma} \{\langle \nabla v_+,\nabla v_+\rangle -(|\sigma|^2+\text{Ric}(N))(v_+)^2\} \,d\Sigma\\
                   &= \int_{\Sigma} \{\langle \nabla v,\nabla v_+\rangle -(|\sigma|^2+\text{Ric}(N)) v v_+\} \,d\Sigma\\
                   &= -\int_{\Sigma} \{( \Delta v)v_+ +(|\sigma|^2+\text{Ric}(N)) v v_+\} \,d\Sigma\\
                   &=0
  \end{align*}
 and similarily  $\mathcal I(v_-,v_-)=0$. 
 As we supposed that $v$ changes sign, there exists $a\in \r$ such that $\int_{\Sigma} (v_+ + a v_-)\, d\Sigma =0.$ So we can use  $\widetilde v:= v_+ + a v_-$ as a test function for stability.  We have
  \begin{equation*}
 \mathcal I(\widetilde v, \widetilde v) =\mathcal  I(v_+,v_+)+2a\mathcal I(v_+,v_-) + a^2 \mathcal I(v_-,v_-)=0.
 \end{equation*}
 Since $\psi$ is stable, we conclude that $\widetilde v$ is a Jacobi function and so is, in particular, smooth.  
 We distinguish two cases:
 
  (i) If $a\neq 1,$ then we conclude that $v_+$ and $v_-$ are smooth and so we can write 
 $Lv=Lv_+ + Lv_-=0$ and 
 $L\widetilde v= Lv_+ + a Lv_-=0.$ It follows that $Lv_+=0$ and so, by the unique continuation principle 
 \cite{Ar},  $v_+$  has to be identically 0, which is a contradiction.
 
 (ii) If $a=1.$ Then $v=\widetilde v$ is a Jacobi function and so satisfies $\frac{\partial{ v}}
 {\partial \nu} = 0$ on $\partial\Sigma.$ Let $p\in \partial\Sigma$ with 
 $\psi(p)\in M\times\{0\}$ and $\Omega$ a neighborhood of $p$ in $\Sigma,$ diffeomorphic to the halfspace $\mathbb R^n_+,$ such that $\psi:\Omega\to M\times[0,l]$ is an embedding. Identify  $S=\psi(\Omega)$ with $\Omega$ and call $\tau$ the reflection in $M\times\mathbb R$ through $M\times\{0\}.$ Then 
 $\overline S:= S\cup \tau(S)$ is a $\mathcal C^2$ hypersurface in $M\times\mathbb R$ with constant mean curvature; $\overline S$  is therefore smooth. We now extend  $v$ to a function $\bar v$ on $\overline S$ by setting $\bar v=0$ on $\tau(S).$ As $\frac{\partial v}{\partial\nu}=0$ on $\partial S,$ we have  $\bar v\in \mathcal C^1(\overline S).$ Furthermore, because $v=\frac{\partial v}{\partial \nu}=0$ on $\partial S,$  one checks that $\bar v$ satisfies the equation $\overline L\bar v=0$ on $\overline S$ in the sense of distributions; $\overline L$ denoting the Jacobi operator on $\overline S.$  By regularity theory for elliptic PDEs, $\bar v$ is smooth. As $\bar v$ vanishes on a non empty subset, the unique continuation principle \cite{Ar} implies that $\bar v \equiv 0$ on $S.$ Again, by the unique continuation principle, one has  $v\equiv 0$ on $\Sigma,$ which is a contradiction.

  We have thus shown that  the function $v$ does not change sign in $\Sigma$. We will assume $v\geq 0,$ the case $v\leq 0$ being similar. The function $v$ satisfies
 \begin{equation*}
\begin{cases}
v&\geq 0,\\
Lv &= 0.
 \end{cases}
\end{equation*}
 As we are assuming $v$ is not identically zero, by the strong maximum principle (Theorem 2.9 in \cite{K}) we know that $v>0$ on the interior of $\Sigma.$ So the interior of $\psi(\Sigma)$ is a local vertical graph. Moreover, since $v$ does not vanish in the interior of $\Sigma,$ it is a first eigenfunction of $L$ and so $0$ is the lowest eigenvalue of $L$
 for the Dirichlet problem on $\Sigma.$ This proves the first statement. 
 
 Let us now  prove the second statement. Assuming $M$ is simply connected and $\psi(\Sigma)$ is not a cylinder, we 
  will show it is globally a vertical graph by analyzing its behaviour near its boundary. 
  
  For the sake of simplicity, we set $M_0=M\times \{0\}$ and $M_1=M\times \{l\}.$ Let $\Gamma_1,\ldots, \Gamma_k$ denote the connected components of $\partial\Sigma.$ By hypothesis, $\psi$ restricted to $\Gamma_i$ is an embedding and so 
 $\psi(\Gamma_i)$ separates  the slice containing it, among $M_0$ and $M_1,$  into two connected components (\cite{ Hirsch}).   
 
Denote by $P: M\times\r \to M$ the  projection on the second factor and set  $F=P\circ \psi.$  Fix $i=1,\ldots,k$  and let $\varphi:\mathcal U\longrightarrow \r$ be a smooth function defined in a neighborhood $\mathcal U\subset M$ of 
$\psi(\Gamma_i)$ such that $\psi(\Gamma_i)=\{\varphi=0\}$ and for each $p\in \Gamma_i, \,\nabla \varphi (\psi(p))=N(p).$ For instance we may take $\varphi$ to be the signed distance function to $\psi(\Gamma_i).$

Let us take  a point 
 $p\in \Gamma_i$, and a curve
  $\gamma:(-\epsilon,0]\to \Sigma$ parametrized by arc-length  with $\gamma(0)=p$ and $\dot\gamma(0)=\nu(p)$. Set  $\overline\gamma(s):=\psi(\gamma(s))=(\alpha(s), t(s))$  where $\alpha =F\circ\gamma:(-\epsilon,0]\to M$ and define
   the real function $h$ on 
  $ (-\epsilon,0]$ by  
  $h(s)=\varphi (\alpha(s)).$
  
  Denoting by $D$ the connection on $M,$ and by $\overline D$ the one on $M\times\r,$ we have
 \begin{align*}
\dot h (0)&=h(0)= 0,\\
 \ddot h(0)&= \nabla^2\varphi(\dot\alpha(0),\dot\alpha(0))+
 \langle \nabla \varphi (\psi(p)), \frac{D\dot\alpha}{ds}(0)\rangle\\
  &= \langle N(p), \frac{\overline D\dot{\overline\gamma}}{ds}(0)\rangle\\
  &=\sigma(\nu(p),\nu(p)).
 \end{align*}
  Note that 
  $$\frac{\partial v}{\partial \nu}=-\sigma(\nu,\nu)\langle\nu,e_{n+1}\rangle=
  \begin{cases}+\sigma(\nu,\nu)\quad {\text{if}}\quad \psi(\Gamma_i)\subset M_0,\\
-\sigma(\nu,\nu)\quad {\text{if}} \quad \psi(\Gamma_i)\subset M_1. 
\end{cases}
$$
By the boundary point maximum principle (see, for instance, Theorem 2.8 in \cite{K}), $\frac{\partial v}{\partial \nu}<0$ on $\partial\Sigma.$
So for $t$ small, $F(\gamma(t))$ lies in the component of $M\setminus F(\Gamma_i)$ which has $N(p)$ 
as outwards (resp. inwards) pointing normal at $F(p)$ if $\psi(\Gamma_i)\subset M_0$ (resp. if $\psi(\Gamma_i)\subset M_1$). It follows that there is a thin strip in the interior of $\Sigma$ around $\Gamma_i$ which projects 
on this component. We call $D_i$ the component of  $M\setminus F(\Gamma_i)$ which does not intersect this projection.

 We let
 $\widetilde\Sigma$ denote the (topological) manifold obtained by attaching, for each $i=1,\dots,k,$  $D_i$ to $\Sigma$ using the diffeomorphism $\psi |_{\Gamma_i}: \Gamma_i\to \partial D_i$  (see, for instance, Theorem 9.29 in \cite{lee} for the details of such a construction).

We define 
$\widetilde F: \widetilde \Sigma \to M$ by
\begin{equation*}
\widetilde F =\begin{cases} F\qquad &{\text{on}}\quad \Sigma\\
\text{identity }\quad &{\text{on}}\quad D_i, i=1,\ldots,k.
\end{cases}
\end{equation*}

Note that $\widetilde F$ is well defined since, by construction of   $\widetilde \Sigma,$ each  $x\in \Gamma_i$ is identified with $\psi(x)\in\partial D_i$. As the interior of $\psi(\Sigma)$ is locally a vertical graph and thanks to the above analysis of the behaviour of $\psi$ near $\partial\Sigma,$ we see that $\widetilde F$ is a local homeomorphism. It is also clear that $\widetilde F$ is a proper map, thus it is a covering map. Therefore
$\widetilde F$ is a global homeomorphism onto $M.$ Consequently, $\psi(\Sigma)$ is a graph over a domain in 
$M$.  In particular, when $M=\r^n, \h^n$ or $\s^n_+$,  Alexandrov's reflection technique shows that $\psi(\Sigma)$ is a
hypersurface of revolution around a vertical axis, see \cite{wente}.
 \end{proof}

\begin{remark} It is clear from the proof that a similar result  is true for stable CMC hypersurfaces 
with free boundary in a half-space $M\times [0,+\infty).$ 
\end{remark}

It is not difficult to deduce from Theorem \ref{product} that if the Ricci curvature of $M$ is bounded from below by a positive constant $\kappa,$ then there is no stable cylinder in $M\times[0,l]$ for $l>\pi/\sqrt \kappa$ (see the argument below). When $M$ has dimension 2, we have the following stronger  statement.

\begin{theorem}\label{thm:nonexistence}
Let  $M$ be an orientable Riemannian surface with Gaussian curvature $K\geq \kappa>0.$  If 
\begin{equation*}
l\geq \frac{4\pi}{\sqrt {3\kappa}}
\end{equation*}
then there is no immersed stable free boundary CMC surface in $M\times \r$ connecting $M\times \{0\}$ to 
$M\times \{l\}.$
\end{theorem}

\begin{proof}

Let $\psi:\Sigma\to M\times[0,l]$ be a stable CMC immersion with free boundary. Suppose $\psi(\Sigma)$ is not a cylinder. By Theorem \ref{thm: free boundary}, we know  that 0 is the lowest eigenvalue of the stability operator $L$ on $\Sigma$ with Dirichlet boundary condition.   Denote by $H$ the mean curvature of $\Sigma$ and by $S$ the scalar curvature function of $M\times \r.$ We have 
$$ 3H^2+S(p,t) =3H^2+K(p)\geq \kappa,\quad \text{ for} \, (p,t)\in M\times\r.$$
We can therefore apply Theorem 1 in  \cite{rosenberg}  which gives an upper  bound for the intrinsic distance in $\Sigma$ to $\partial\Sigma$.  More precisely, for each $x\in \Sigma,$ 
$$ {d_{\Sigma}}(x,\partial\Sigma)\leq \frac{2\pi}{\sqrt {3 \kappa}}.$$

If $\psi:\Sigma\to M\times [0,l]$ connects   $M\times \{0\}$ to 
$M\times \{l\},$ take $x\in \Sigma$ such that $\psi(x)\in M\times \{\frac{l}{2}\}.$ Then
$$\frac{l}{2}\leq {d_{\Sigma}}(x,\partial\Sigma)\leq \frac{2\pi}{\sqrt {3 \kappa}}.$$

Moreover, if we have the equality  $l/2=2\pi/ \sqrt{3\kappa},$ then ${d_{\Sigma}}(x,\partial\Sigma)=l/2$
and so at least one of the geodesic segments $\{x\}\times [0,l/2]$ or $\{x\}\times [l/2,l]$ is contained in $\Sigma.$ In this case, by Theorem \ref{thm: free boundary}, $\psi(\Sigma)$ has to be a cylinder. 

To conclude we need to check that cylinders over closed constant geodesic curves are unstable for $l\geq 4\pi/{\sqrt{3\kappa}}.$ In fact a stronger statement is true.  Consider  a  curve $\gamma$ of constant geodesic  curvature  $h$ in $M.$ The lowest eigenvalue $\lambda_1(\gamma)$ of the stability operator
on $\gamma$ satisfies
$$\lambda_1(\gamma)\leq \frac{\mathcal I(1,1)}{\int_{\gamma} ds} =-\frac{\int_{\gamma} (h^2+K)ds}{\int_{\gamma}ds} \leq -\kappa.$$
For $l>\pi/{\sqrt \kappa}$,  we get 
$$\lambda_1(\gamma)+\frac{\pi^2}{l^2} \leq-\kappa + \frac{\pi^2}{l^2}<0$$
and so, according to Theorem  \ref{product}, the cylinder is unstable. This completes the proof. 
\end{proof}


\section{Stable capillary surfaces of genus zero in certain warped products}\label{sec:warped}

In this section we are interested in stable capillary surfaces in   warped product spaces  of the type  $[0,l]\times_f M$ where $M=\r^2, \s^2$ or $\h^2$ and $f:[0,l]\to \r$ is a smooth positive function. More precisely, $[0,l]\times_f M$ is  the product 
space $[0,l]\times M$ endowed with the metric (with some abuse of notation)
$$g=dt^2 + f(t)^2 g_0,$$ 
 $g_0$ being the metric on $M.$ 

An important feature of these spaces we will use  is that if $\varphi: M\longrightarrow M$ is an isometry, then its trivial extension $\{\text{id}\} \times\varphi : [0,l]\times_f M\longrightarrow
[0,l]\times_f M$  is an isometry too.

We have the following extension to these spaces of a theorem proved in the Euclidean case in \cite{A-S}.

 
\begin{theorem}\label{thm: warped} 
Let $\psi:\Sigma\to [0,l]\times_f M$ be an immersed capillary surface of genus zero connecting the two boundary components of $[0,l]\times_f M$ where $M=\r^2,\h^2,$ or $\s^2$ and having constant contact angles $\theta_0$ and $\theta_1$ with $\{0\}\times M$ and $\{l\}\times M,$ respectively.

If $\psi$ is stable then $\psi(\Sigma)$ is a surface of revolution around an axis $\r\times \{x\},\, x\in M$. 
\end{theorem} 
\begin{proof}
Let  $\gamma$  be a connected component of  $\partial\Sigma$  such that $\psi(\gamma)$  lies on
  $\{ 0\}\times M.$ We claim there is a circle $\mathcal C$ in $M$ bounding a disk containing $\psi(\gamma)$ and touching
$\psi(\gamma)$ at least at 2 points. For $M=\r^2,$ the circumscribed circle about $\psi(\gamma)$ has this property (cf.  \cite{osserman}). When $M=\h^2,$  we can just take the (Euclidean) circumscribed circle about 
  $\psi(\gamma)$ in the half-space model since  Euclidean circles in this model are also
  metric circles for the hyperbolic metric. When $M=\s^2,$
 we pick a point $p\in\s^2\setminus \psi(\gamma)$  and perform stereographic projection $\sigma$ from the point $p$ onto the Euclidean plane $\r^2. $ Let    $\mathcal C^\ast$ be the circumscribed  circle  in $\r^2$
about  $\sigma(\psi(\gamma)).$ Its inverse image $\mathcal C:=\sigma^{-1}(\mathcal C^\ast)$  is a circle in $\s^2$
which has the required property. 

 We will show that $\psi(\Sigma)$ is a surface of revolution around the vertical axis passing through the center $x$ of $\mathcal C.$ In the case of $\s^2,$ we take $x$ to be the center of the disk bounded by $\mathcal C$ that contains $\psi(\gamma).$ 

 Let us  consider the  Jacobi function $u$ on $\Sigma$ induced by 
the rotations around the vertical axis passing through $x.$ The function $u$ verifies
\begin{equation}\label{rotation}
\begin{cases}  Lu = 0 \,\,\,\,\quad {\text {on}}\quad \Sigma\\
\frac{\partial u}{\partial \nu} = q\, u \quad {\text{on}}\quad \partial \Sigma
 \end{cases}
\end{equation}

We will prove  that $u\equiv 0$ on $\Sigma.$ 

Suppose, by contradiction, $u$ is not identically zero. Then its nodal set $u^{-1}(0)$ in the interior of $\Sigma$ has the structure of a graph (cf. \cite{cheng}). We will show that $u$ has at least 3 nodal domains by analyzing the set   $u^{-1}(0)\cap\gamma$;  a nodal domain of $u$ being a connected component of $\Sigma\setminus u^{-1}(0)$.

We first note  that, because of the boundary condition satisfied by $u,$ if $p\in u^{-1}(0)\cap \partial\Sigma$ then
$\frac{\partial u}{\partial \nu}(p)=0.$ It follows from the boundary point maximum principle (see Theorem 2.8 in \cite{K}) that  $u$ changes sign in any neighborhood of $p\in u^{-1}(0)\cap \partial\Sigma$    unless $u$ is identically zero in a neighborhood of $p.$ In the latter case, by the unique continuation principle \cite{Ar}, $u$ would vanish everywhere on $\Sigma,$ contradicting our assumption. Consequently  each such point  
lies on the boundary of at least 2 components of the  set $\{u\neq 0\}.$

If $p\in \partial\Sigma$ is a critical point of the distance function to $x$ restricted to $\gamma$, then one can check that $u(p)=0$.
Now, we observe that by the choice of $\mathcal C,$ there are at least 3 points in  $u^{-1}(0)\cap \gamma.$ Indeed, we already know  that there are two points  in $u^{-1}(0)\cap \gamma.$ A third one is a point of $\gamma$ whose image by $\psi$ is a closest one to the center $x$ 
of $\mathcal C.$ 

 Since by hypothesis $\Sigma$ is topologically a planar domain, using the above information and the Jordan curve theorem, it is easy to see this implies that $u$ has at least 3 nodal domains.

 Denote by $\Sigma_1$ and $\Sigma_2$ two of these components and consider the following function in the Sobolev space $H^1(\Sigma)$:
 \begin{equation*}
\widetilde u ={ \begin{cases}
\quad u \,\,\,\,\,\quad{\text{on}}\quad \Sigma_1\\
\alpha \, u \quad\quad {\text{on}}\quad \Sigma_2\\
\quad 0\qquad {\text{on}}\quad \Sigma\setminus (\Sigma_1\cup \Sigma_2)
 \end{cases}}
\end{equation*} 
 where $\alpha\in \r$ is chosen so that $\int_{\Sigma}\widetilde u\,dA =0.$
 Using (\ref{rotation}) we compute
 \begin{align*}
 \int_{\Sigma_1}\{ \langle \nabla\widetilde u, \nabla\widetilde u\rangle -(|\sigma|^2+\text{Ric}(N)) {\widetilde u}^2 \}dA
 &= \int_{\Sigma_1}\{ \langle  \nabla u,\nabla \widetilde u\rangle -(|\sigma|^2+\text{Ric}(N)) \,u{\widetilde u}\}dA\\
 = -\int_{\Sigma_1} ( \Delta u+ (|\sigma|^2+\text{Ric}(N)) u) \widetilde u \,dA  &+\int_{\partial\Sigma_1} \widetilde u\frac{\partial u}{\partial \nu}\,ds\\
 &= \int_{{\partial \Sigma_1}\cap \partial\Sigma} q {\widetilde u}^2 \,ds
  \end{align*}
  Using a similar computation on $\Sigma_2,$ we deduce that 
  $\mathcal  I(\widetilde u,\widetilde u)=0.$
 As $\Sigma$ is stable, we conclude that  $\widetilde u$ is a Jacobi function. Indeed, the quadratic form on $\mathcal F$ associated to $\mathcal I$ has a minimum at $\widetilde u$ and so $\widetilde u$ lies in the kernel of $\mathcal I.$ However, $\widetilde u$ vanishes on a 
  non empty open set. By the unique continuation principle \cite{Ar}, $\widetilde u$ has to vanish everywhere, which is a contradiction. 
  
  Therefore $u\equiv 0.$ This means that $\psi(\Sigma)$ is a surface of revolution around the axis through $x.$
 \end{proof}

The above result applies to 
the region bounded by two spheres centered at the origin for  metrics on $\r^3$ that are invariant under the group $SO(3).$  This is, for instance, the case for the region bounded by two concentric spheres in 
$\r^3, \h^3$ or $\s^3.$  Another interesting case to which the result applies is  the region bounded by two parallel horospheres  in $\h^3.$ Indeed,  consider the half-space model of $\h^3,$ that is, 
$$\h^3=\{(x_1,x_2,x_3)\in \r^3, x_3 >0\}$$
equipped  with the metric
 $$ds^2=\frac{dx_1^2+dx_2^2+dx_3^2}{x_3^2}.$$  Making the change of variables  $t=\log x_3,$ we see that the metric writes 
$$ds^2= e^{-2t} (dx_1^2+dx_2^2)+dt^2.$$
Otherwise said, $\h^3$ can be viewed as the warped product $\r \times_{e^{-t}} \r^2$  and we may assume, up to a rigid motion, that the two parallel horospheres are slices  $\{x_3= \text{cst}\}.$ We can thus state the following
\medskip
\begin{corollary}\label{cor:horospheres} 
Let $\psi:\Sigma\to \h^3$ be an immersed capillary surface of genus zero connecting two parallel horospheres.  If $\psi$ is stable then $\psi(\Sigma)$ is a surface of revolution around an axis orthogonal to the horospheres. 
\end{corollary} 

 The isoperimetric problem in regions bounded by two parallel horospheres in $\h^3$  was studied in \cite{chaves} by Chaves, da Silva and Pedrosa.


\section{Stable closed CMC  surfaces in a product $M\times\s^1(r)$}\label{sec:closed}
In this section, we consider  stable CMC surfaces in a product $M\times \s^1(r),$ where $M$ is a Riemannian surface with curvature positively bounded from below, $K\geq \kappa,$ with $\kappa >0$ a constant and $\s^1(r)$ is the circle of radius $r>0.$  We have the following result which is in the same vein as Theorem \ref{thm:nonexistence}.

\begin{theorem}\label{thm:closed}
Let $\psi: \Sigma\to M\times \s^1(r)$ be a  stable CMC  immersion of a closed connected and oriented surface $\Sigma$ in $M\times \s^1(r)$ where $M$ is an orientable  surface with Gaussian curvature satisfying $K\geq \kappa>0,$ for some constant $\kappa,$ and $\s^1(r)$ is the circle of radius $r>0.$  Denote by $p:M\times\s^1(r)\to\s^1(r)$ the canonical projection. Suppose that
$$r\geq\frac{4}{\sqrt {3\kappa}},$$
then the induced homomorphism
$$(p\circ \psi)_{\ast} : \pi_1(\Sigma)\to \pi_1(\s^1(r))$$
is trivial and therefore 
 $\psi$ lifts to a stable CMC immersion $\widetilde\psi: \Sigma \to M\times \r.$
\end{theorem}

\begin{proof}

We first derive an upper bound for the diameter of $\Sigma.$ Denote by $d_{\Sigma}$  the intrinsic distance on $\Sigma$ and by $\text {diam} (\Sigma)$ its diameter and let $x,y\in \Sigma$ be such that $d:=d_{\Sigma}(x,y)=\text {diam} (\Sigma).$ Then the  disks 
$B_{d/2}(x)$ and $B_{d/2}(y)$ of radius $d/2$ centered at $x$ and $y,$ respectively, have disjoint interiors. Denote by $\lambda_1(B_{d/2}(x))$ (resp. $\lambda_1(B_{d/2}(y))$)  the first eigenvalue of the operator $L$ with Dirichlet boundary condition on $B_{d/2}(x)$ (resp. on  $B_{d/2}(y)$).  Since $\psi$ is stable, the second eigenvalue of $L$ on $\Sigma$ verifies $\lambda_2(\Sigma)\geq 0$ (cf. Theorem \ref{criterion}). By the min-max characterization of the eigenvalues of $L$, we have
$$\lambda_2(\Sigma)=
 \inf_{{E\subset H^1(\Sigma),
 \text{dim} E=2}}\,\,
  \sup_{u\in E\setminus
 \{0\}} \frac{-\int_{\Sigma}u Lu\,d\Sigma}{\int_{\Sigma} u^2\, d\Sigma}.
  $$
 Since $B_{d/2}(x)$ and $B_{d/2}(y)$ have disjoint interiors, it follows  that 
$$\lambda_2(\Sigma)\leq \max \{ \lambda_1(B_{d/2}(x)), \lambda_1(B_{d/2}(y))\}.$$
Therefore at least one of the numbers $\lambda_1(B_{d/2}(x))$ and $\lambda_1(B_{d/2}(y))$ is  nonnegative.  Denoting by $H$ the mean curvature of $\Sigma$ and by $S$ the scalar curvature function of $M\times\s^1(r),$ we have
$$ 3H^2+S(p,\theta)=3H^2+K(p)\geq \kappa, \quad  \text{for}\,(p,\theta)\in M\times\s^1(r).$$
We can thus apply Theorem 1 in  \cite{rosenberg} which gives the upper bound 
$d/2\leq 2\pi/\sqrt{3\kappa},$ that is, 
\begin{equation}\label{ineqdiam1}
\text{diam}(\Sigma)\leq \frac{4\pi}{\sqrt{3\kappa}}.
\end{equation}

Suppose now  that the induced homomorphism $\pi_1(\Sigma)\to \pi_1(\s^1(r))$ is not trivial and consider  a  loop  $\gamma:\s^1\to \Sigma$ such that the loop $p\circ \psi\circ \gamma:\s^1\to \s^1(r)$  is nontrivial. We may assume that $\gamma$ is a piecewise $\mathcal C^1$-immersion. Denote by $d_0$ the distance on $M\times\s^1(r),$ 
by $d_{\Sigma}$ the one on $\Sigma$ and by $d_1$ the one on $\s^1(r).$ 
Let $x\in \gamma(\s^1)$, then $p\circ\psi \circ \gamma(\s^1)$
 is not entirely contained  in an interval of radius $<\pi r$ centered at $p(\psi(x))$ in $\s^1(r),$ because otherwise the loop $p\circ\psi\circ\gamma$  would be trivial.  Therefore, there  exists a point $y\in\gamma(\s^1)$ such that  
\begin{equation}\label{ineqdiam2}
\pi r= d_1(p(\psi(x)),p(\psi(y))\leq d_0(\psi(x),\psi(y)\leq d_{\Sigma}(x,y)\leq {\text{diam}}(\Sigma).
\end{equation}
It follows from (\ref{ineqdiam1}) and (\ref{ineqdiam2}) that $r\leq 4/\sqrt{3\kappa}$.

 Moreover, if $r= 4/\sqrt{3\kappa},$ then 
all the inequalities  in  (\ref{ineqdiam2}) are equalities. In particular $\pi r=d_1(p(\psi(x)),p(\psi (y)))= d_0(\psi(x),\psi(y)).$ It follows that if $\psi(x)=(x_0,\theta)\in M\times \s^1(r)$ then, we necessarily have $ \psi(y)=(x_0,\theta^\ast),$ where $\theta^\ast \in \s^1(r)$ is the antipodal point of $\theta.$  Now, we can deform slightly the loop $\gamma$ to a 
homotopic one $\bar\gamma$ satisfying $x\in \bar\gamma(\s^1)$ and $\psi(y)\notin \psi(\bar\gamma(\s^1)).$ By the same argument as before, there must exist a point $\bar y\in \bar\gamma(\s^1)$ satisfying $\pi r=d_1(p(\psi(x)),p(\psi(\bar y)))= d_0(\psi(x),\psi(\bar y)).$ Again this implies that
$\psi(\bar y)=(x_0,\theta^\ast)=\psi(y),$ which is   a contradiction. Therefore $r<{4}/{\sqrt {3\kappa}}$ if the induced homomorphism $\pi_1(\Sigma)\to \pi_1(\s^1(r))$ is not trivial.
 This completes the proof.
\end{proof}

 As an application, consider the sphere $\s^2$ endowed with a Riemannian metric $g$ of Gaussian curvature $K\geq \kappa>0,$  for some positive constant $\kappa$  and let $r\geq {4}/{\sqrt {3\kappa}}.$ Under these hypotheses, it follows from  Theorem \ref{thm:closed} and Theorem 3 in \cite{ros-meso} that an isoperimetric region in $\s^2(g)\times\s^1(r)$ is either a slab or a domain bounded by a surface of genus $\leq 2.$ Indeed, the boundary $\Sigma$ of an isoperimetric domain $\Omega$  in  $(\s^2,g)\times \s^1(r)$ lifts to an embedded stable 
closed CMC surface in $(\s^2,g)\times\r$ by our result. By Theorem 3 in \cite{ros-meso},  $\Sigma$  has genus at most 2 if it is connected and is a union of a finite number of horizontal slices if it is disconnected. 
 In the latter case,  $\Omega$ has to be a finite union of disjoint slabs $\s^2\times [\theta_i,\theta_i+\alpha_i], i=1,\dots,n,$ in $\s^2\times \s^1(r).$ Note that a slab of width $\sum_{i=1}^n\alpha_i$ has the same volume as the union $\mathcal U=\cup_{i=1}^{n}\s^2\times [\theta_i,\theta_i+\alpha_i]$ and its boundary has area less than the area of $\partial \mathcal U.$ As $\Omega$ is an isoperimetric domain, it has to be a slab in this case. 


\end{document}